\documentclass[11pt]{article}
\usepackage{amsmath,amssymb,amsthm,eucal}
\usepackage{color}
\usepackage[all]{xy}

\title{\vspace*{-1pc}%
       Dense domains, symmetric operators and spectral triples}

\author{Iain Forsyth\dag, Bram Mesland\S, Adam Rennie\ddag, 
\thanks{email: 
\texttt{iain.forsyth@anu.edu.au}, \texttt{b.mesland@warwick.ac.uk}, \texttt{renniea@uow.edu.au}
}
\\[3pt]
\dag Mathematical Sciences Institute, Australian National University,\\  Canberra, Australia
\\[3pt]
\S Mathematics Institute,
Zeeman Building, University of Warwick,\\
Coventry CV4 7AL, UK
\\[3pt]
\ddag School of Mathematics and Applied Statistics, University of Wollongong\\
Wollongong, Australia\\
}

\usepackage{amsmath}
\usepackage{color}
\usepackage{amssymb}
\usepackage{amsfonts}
\usepackage{amsthm}
\usepackage{amscd}
\usepackage{latexsym}
\usepackage[pdftex]{hyperref} 
\usepackage{verbatim}

\newcommand{\bma}{\left(\begin{array}{cc}}
\newcommand{\ema}{\end{array}\right)}
\newcommand{\bca}{\left(\begin{array}{c}}
\newcommand{\eca}{\end{array}\right)}

\advance\topmargin by -\headheight
\advance\topmargin by -\headsep
\evensidemargin=\oddsidemargin
\makeatletter
\def\section{\@startsection{section}{1}{\z@}{-3.5ex plus -1ex minus
  -.2ex}{2.3ex plus .2ex}{\large\bf}}
\def\subsection{\@startsection{subsection}{2}{\z@}{-3.25ex plus -1ex
  minus -.2ex}{1.5ex plus .2ex}{\normalsize\bf}}
\makeatother

\numberwithin{equation}{section} 

\theoremstyle{plain} 
\newtheorem{thm}{Theorem}[section]
\newtheorem{lemma}[thm]{Lemma}
\newtheorem{prop}[thm]{Proposition}
\newtheorem{corl}[thm]{Corollary}

\theoremstyle{definition} 
\newtheorem{defn}[thm]{Definition}

\theoremstyle{remark} 

\newcommand{\A}{\mathcal{A}}  
\newcommand{\B}{\mathcal{B}}  
\newcommand{\C}{\mathbb{C}}   
\newcommand{\D}{\mathcal{D}}  
\renewcommand{\H}{\mathcal{H}}  
\newcommand{\J}{\mathcal{J}}  
\newcommand{\K}{\mathcal{K}}  
\newcommand{\R}{\mathbb{R}}   
\newcommand{\Z}{\mathbb{Z}}   
\newcommand{\stroke}{\mathbin|}     

\def\pairL_#1(#2|#3){{}_{#1}(#2\stroke#3)} 
\def\pairR(#1|#2)_#3{(#1\stroke#2)_{#3}} 
\def\scal<#1|#2>{\langle#1\stroke#2\rangle} 

\newbox\ncintdbox \newbox\ncinttbox 
	\setbox0=\hbox{$-$}
	\setbox2=\hbox{$\displaystyle\int$}
	\setbox\ncintdbox=\hbox{\rlap{\hbox
		to \wd2{\hskip-.125em \box2\relax\hfil}}\box0\kern.1em}
	\setbox0=\hbox{$\vcenter{\hrule width 4pt}$}
	\setbox2=\hbox{$\textstyle\int$}
	\setbox\ncinttbox=\hbox{\rlap{\hbox
		to \wd2{\hskip-.175em \box2\relax\hfil}}\box0\kern.1em}

\newcommand{\de}{\partial}
\newcommand{\Th}{^\text{\textnormal{th}}}
\DeclareMathOperator{\dom}{dom}
\newcommand{\Ideal}[1]{\left\langle#1\right\rangle}
\newcommand{\Ket}[1]{\left|#1\right\rangle}
\newcommand{\Bra}[1]{\left\langle#1\right|}
\DeclareMathOperator{\Span}{span}
\newcommand{\ol}{\overline}

\begin{document}

\maketitle

\vspace{-2pc}

\begin{abstract}
This article is about erroneous attempts to weaken the standard definition of 
unbounded Kasparov module (or spectral triple).  We present counterexamples 
to claims in the literature that Fredholm modules can be obtained from 
these weaker variations of spectral triple. Our counterexamples are constructed using
self-adjoint extensions of symmetric operators.
\end{abstract}

%

\tableofcontents

\parskip=4pt
\parindent=0pt

\addtocontents{toc}{\vspace{-1pc}}

\section{Introduction}
\label{sec:intro}

In this note we show, by counterexample, that weaker definitions of unbounded Kasparov module, and so spectral triple, may not yield $KK$ or $K$-homology classes. In particular, we consider counterexamples arising from extensions of symmetric operators. These counterexamples address
 errors both in \cite[pp 164-165]{Bl}
and subsequent errors in \cite{Valette}.


%

Recently we found that one of the standard texts on $KK$-theory, \cite{Bl}, is overly
ambitious in extending the definition of unbounded Kasparov module. 
The principal requirement of any definition of unbounded Kasparov module 
is that it defines a $KK$-class. This requirement constrains how far the definition
can be extended. The work of Baaj-Julg, \cite{BJ}, provides sufficient conditions for this to be guaranteed. Different conditions apply to the definition of relative Fredholm modules, which can be obtained from symmetric operators, as shown by \cite{BDT}.



The definition of spectral triple that does give a well defined Fredholm module 
reads as follows (see \cite{BJ, CP1} and section 2 of the present paper):
\begin{defn} 
\label{def:ST}
A  
spectral triple $(\A,\H,\D)$ is given by a Hilbert space $\H$, a
$*$-subalgebra  $\A\subset\, \B(\H)$
acting on
$\H$, and a densely defined unbounded self-adjoint operator $\D$  such that:

1. $a\cdot{\rm dom}\,\D\subset {\rm dom}\,\D$ for all $a\in\A$, so that
$[\D,a]$ is densely defined.  Moreover, $[\D,a]$ is bounded on ${\rm dom}\,\D$ and so extends to a bounded operator in
$\B(\H)$ for all $a\in\A$;

2. $a(1+\D^2)^{-1/2}\in\K(\H)$ for all $a\in\A$.
\smallskip

We say that $(\A,\H,\D)$ is even if in addition there is a $\mathbb{Z}_2$-grading
such that $\A$ is even and $\D$ is odd. This means there is an operator $\gamma$ such that
$\gamma=\gamma^*$, $\gamma^2={\rm Id}_\H  
$, $\gamma a=a\gamma$ for all $a\in\A$ and
$\D\gamma+\gamma\D=0$. Otherwise we say that $(\A,\H,\D)$ is odd.
\end{defn}

It is asserted in \cite[pp 164-165]{Bl} that condition 1. of the definition may be weakened to
\begin{enumerate}
\item[1.']\label{wrong}  there is a subspace $Y$ of ${\rm dom}\,\D$ such that $Y$ is dense in $\H$, 
$a\cdot Y\subset{\rm dom}\,\D$, and $[\D,a]$ is bounded on $Y$.
\end{enumerate}
Moreover \cite[Proposition 17.11.3]{Bl} asserts that condition \ref{wrong}' ensures
that $(\A,\H,\D(1+\D^2)^{-1/2})$ is a Fredholm module.
Our first and fourth counterexamples prove that this is false, 
by showing that if the algebra $\A$
does not preserve the domain of $\D$, then the commutators $[\D(1+\D^2)^{-1/2},a]$ need
not be compact, even when $(1+\D^2)^{-1/2}$ is compact. After writing this work, the
paper \cite{Hil} was brought to our attention. In \cite[Section 4]{Hil},
Hilsum provides a $K$-theoretic
contradiction of Blackadar's result.

Unfortunately the problems in \cite{Bl} have contributed to further errors in the literature.
In \cite[Theorems 1.2, 1.3, 6.2]{Valette}, the authors assert
that a Fredholm module can be obtained from any
self-adjoint extension of a symmetric operator $\D$ satisfying certain spectral-triple-like conditions,
\cite[Definition 1.1, Definition 6.3]{Valette}. 
They further claim that
the resulting $K$-homology class is independent of the particular self-adjoint extension.
Both these claims are false, as our counterexamples show.


%
{\bf Acknowledgements}. We would like to thank Alan Carey for useful 
discussions at an early stage of this project. The first and third authors 
were supported by the Australian Research Council, while the 
second author was supported by the EPSRC grant EP/J006580/2.

\section{From spectral triple to Fredholm module}

The idea of the (hard part of the) proof that a spectral triple $(\A,\H,\D)$
defines a Fredholm module, 
due originally to Baaj and Julg, \cite{BJ}, is to write, for $a\in\A$,
\begin{equation}
[\D(1+\D^2)^{-1/2},a]=[\D,a](1+\D^2)^{-1/2}+\D[(1+\D^2)^{-1/2},a].
\label{eq:Leibniz}
\end{equation}
As we want to show that the left hand side is compact, the aim is to show that
both terms on the right are compact. 
For the second term, one writes
$$
(1+\D^2)^{-1/2}=\frac{1}{\pi}\int_0^\infty \lambda^{-1/2}(1+\lambda+\D^2)^{-1}\,d\lambda,
$$
then takes the commutator with $a$ and multiplies by $\D$ yielding
\begin{equation}
\D[(1+\D^2)^{-1/2},a]=\frac{1}{\pi}\D\int_0^\infty \lambda^{-1/2}[(1+\lambda+\D^2)^{-1},a]\,d\lambda.
\label{eq:int-formula}
\end{equation}
A careful analysis of the naive equality
\begin{align}
&\D[(1+\D^2)^{-1/2},a]=\frac{-1}{\pi}\int_0^\infty \lambda^{-1/2}\times\nonumber\\
&\Big(\D^2(1+\lambda+\D^2)^{-1}[\D,a](1+\lambda+\D^2)^{-1}+\D(1+\lambda+\D^2)^{-1}[\D,a]\D(1+\lambda+\D^2)^{-1}\Big)\,d\lambda
\label{eq:naive}
\end{align}
appears in \cite[Lemmas 2.3 and 2.4]{CP1}. 
There, and in the intervening remarks, it is proved
that this equality is valid when $a$ preserves the domain of $\D$.
A similar analysis, employing the Cauchy integral formula, appears in \cite[Proposition 1.1]{BDT}.
The remainder of the proof is to show that the right hand side of Equation \eqref{eq:naive} is
a norm convergent integral with compact integrand, thus showing that the left hand side is compact.

The proof of \cite[Lemma 2.3]{CP1} makes it clear that the equality \eqref{eq:naive}
requires careful domain considerations, and that  \eqref{eq:naive} does not
hold simply for algebraic reasons.

Thus we see that the Baaj-Julg approach to proving compactness
of $[\D(1+\D^2)^{-1/2},a]$ using Equations \eqref{eq:Leibniz} 
and \eqref{eq:int-formula} requires the 
assumption that 
$a$ preserves the domain of 
$\D$. As a slight generalisation, it is asserted in \cite{CP1} that the Baaj-Julg proof 
can be pushed through provided $a$ maps a core for $\D$ into the domain of 
$\D$. We amplify on this in the next Proposition.

\begin{prop}
Let $\D:\dom\D\subset\H\rightarrow\H$ be a 
closed operator, let $X\subset\dom\D$ be a core for $\D$, and let $a\in \B(\H)$ satisfy
\begin{itemize}
\item[(1)]
$a\cdot X\subset\dom\D$, and
\item[(2)]
$[\D,a]:X\rightarrow\H$ is bounded on $X$ and so extends to an operator in $\B(\H)$.
\end{itemize}
Then $a\cdot\dom\D\subset\dom\D$ so that $[\D,a]:\dom\D\to\H$ is well-defined. 
If moreover there is an $\H$-norm dense subspace 
$Y\subset\dom\D^*$ such that $a^*\cdot Y\subset\dom\D^*$, 
then $[\D,a]:\dom\D\rightarrow\H$ extends to an operator in $\B(\H)$.
\end{prop}

\begin{proof}
Since $X$ is a core for $\D$, it is dense in $\dom\D$ in the graph norm. 
Let $x\in\dom\D$, and choose a sequence $\{x_n\}_{n=1}^\infty\subset X$ 
such that $x_n\rightarrow x$ in the graph norm, which is equivalent to 
$x_n\rightarrow x$ and $\D x_n\rightarrow \D x$ in the usual norm. 
Since $a\in \B(\H)$, $ax_n\rightarrow ax$, and $\{\D ax_n\}_{n=1}^\infty$ 
is Cauchy in the usual norm since
\begin{align*}
\|\D ax_n-\D a x_m\|&=\|a \D x_n - a \D x_m+ [\D,a] x_n-[\D,a] x_m\|\\
&\leq\|a\|\|\D x_n-\D x_m\|+\|[\D,a]\|\|x_n-x_m\|\rightarrow0.
\end{align*}
Hence $\{ax_n\}_{n=1}^\infty$ is Cauchy in the graph norm, and since $\D$ 
is closed,  there is some $y\in\dom\D$ such that $ax_n\rightarrow y$ in the 
graph norm. This implies that $ax_n\rightarrow y$ in the usual norm, and 
since $ax_n\rightarrow ax$ in the usual norm we see that $y=ax$. Hence $ax\in\dom\D$.

Now suppose that $Y\subset\dom\D^*$, $a^*\cdot Y\subset\dom\D^*$. 
To show that $[\D,a]:\dom\D\rightarrow\H$ is bounded, it is enough to show that 
$[\D,a]$ is closeable, since then $\ol{[\D,a]}\supset\ol{[\D,a]|_X}$ 
which is everywhere defined and bounded. Let $\xi\in\dom\D$ and $\eta\in Y$. Then
\begin{align*}
\Ideal{[\D,a]\xi,\eta}=\Ideal{a\xi,\D^*\eta}-\Ideal{\D\xi,a^*\eta}=\Ideal{\xi,a^*\D^*\eta}-\Ideal{\xi,\D^* a^*\eta}=\Ideal{\xi,-[\D^*,a^*]\eta}.
\end{align*}
Hence $\dom([\D,a])^*\supset Y$. Since $[\D,a]$ is closeable if and only if 
$([\D,a])^*$ is densely defined, if $Y$ is dense in $\H$ then $[\D,a]$ is closeable and thus extends to an operator in $\B(\H)$.
\end{proof}

\begin{corl}
Condition 1. of Definition \ref{def:ST} is equivalent to
\begin{itemize}
\item[i.]
for all $a\in\A$ there exists a core $X$ for $\D$ such that $a\cdot X\subset\dom\,\D$, 
and such that $[\D,a]:X\rightarrow\H$ is bounded on $X$.
\end{itemize}
\end{corl}

To simplify some later computations with bounded transforms $F=\D(1+\D^2)^{-1/2}$
of unbounded self-adjoint operators, we include the following elementary Lemma.

\begin{lemma} 
\label{lem:compact-comms}
Let $\D$ be an unbounded self-adjoint operator on the Hilbert space $\H$, and suppose
that $(1+\D^2)^{-1/2}$ is compact. Then with $F=\D(1+\D^2)^{-1/2}$, $P_+=\chi_{[0,\infty)}(\D)$,
$P_-=1-P_+$, and $A\subset \B(\H)$ a $C^*$-algebra,
the operator $[F,a]$ is compact for all $a\in A$ if and only if $P_{+}aP_{-}$ is compact for all $a\in A$.
\end{lemma}
\begin{proof}
The phase of $\D$ is 
$$
\textnormal{Ph}(\D)=P_+-P_-,
$$
and is a compact perturbation of 
$F=\D(1+\D^2)^{-1/2}$, so for $a\in A$, the commutator $[F,a]$ is compact if and only if 
$[\textnormal{Ph}(\D),a]$ is compact. Since $P_{+}+P_{-}=1$, we see that
\begin{align*}
[\textnormal{Ph}(\D),a]=(P_{+}+P_{-})[\textnormal{Ph}(\D),a](P_{+}+P_{-})
&=2P_+aP_--2P_-aP_+,
\end{align*}
so that $[\textnormal{Ph}(\D),a]$ is compact if and only if $P_+aP_--P_-aP_+$ is compact. 
If $P_+aP_--P_-aP_+$ is compact, then so are
$$
P_+(P_+aP_--P_-aP_+)=P_+aP_-\quad\text{and}\quad -P_-(P_+aP_--P_-aP_+)=P_-aP_+,
$$
so $[F,a]$ is compact if and only if $P_+aP_-$ and $P_-aP_+$ 
are compact. Since $(P_+aP_-)^*=P_-a^*P_+$, we have
$[F,a]$ is compact for all $a\in A$ if and only if 
$P_+aP_-$ is compact for all $a\in A$.
\end{proof}

\section{The counterexamples}
\label{sec:background}

In this section we produce counterexamples to  statements appearing 
in \cite[Theorems 1.2, 1.3, 6.2]{Valette}. The first and fourth of these 
counterexamples also show
that the definition of spectral triple using condition 1.' in place of condition 1. does 
not guarantee that we obtain a Fredholm module.
\subsection{Finite deficiency indices: the unit interval}
\label{subsec:finite}
Initially, the authors of \cite{Valette} confine their attention to 
symmetric operators with equal and finite deficiency indices, 
\cite[Definition 1.1, Theorem 1.2]{Valette}. We begin with our counterexample to
their claims that a Fredholm module is obtained from any self-adjoint extension
of such an operator (which must also satisfy spectral-triple-like conditions). Our extension
will also satisfy the definition of spectral triple using condition 1.'.
In particular, \cite[Proposition 17.11.3]{Bl} and
\cite[Theorem 1.2]{Valette} are false. 

The basic properties of the following example are worked out in \cite[Volume I]{RS}.
Let $\H=L^2([0,1])$ and let $AC([0,1])$ be the absolutely continuous functions. Set
$$
{\rm dom}\,\D=\{f\in AC([0,1]):\, f'\in L^{2}([0,1]),\  f(0)=f(1)=0\},
\quad\D=\frac{1}{i}\frac{d}{dx},
$$
so that $\D$ is a closed symmetric operator with adjoint
$$
{\rm dom}\,\D^*=\{f\in AC([0,1]):\,f'\in L^{2}([0,1])\},\quad\D^*=\frac{1}{i}\frac{d}{dx}.
$$
The deficiency indices of $\D$ are both 1. The operator $\D^*\D$ has normalised eigenvectors
$$
\D^*\D\left(\sqrt{2}\sin(\pi n x)\right)=\pi^2n^2\sqrt{2}\sin(\pi n x),\quad n\in\mathbb{Z},
$$
which are known to be complete for $L^2([0,1])$. 
Since $n^2\pi^2\rightarrow\infty$ as $|n|\rightarrow\infty$, it follows that
$$
(1+\D^*\D)^{-1/2}\in \K(\H).
$$
It is clear that $C^\infty([0,1])$ preserves both ${\rm dom}\,\D$ 
and ${\rm dom}\,\D^*$, and that $[\D^*,a]$ is bounded for all $a\in C^\infty([0,1])$. 
In particular, the data $(C^\infty([0,1]), L^2([0,1]),\D)$ satisfy \cite[Definition 1.1]{Valette}.
Let $\D_0$ be the self-adjoint extension defined by
$$
{\rm dom}\,\D_0=\{f\in AC([0,1]):\,f'\in L^{2}([0,1]), f(0)=f(1)\}.
$$ 
The eigenvectors of $\D_0$ are
$$
\D_0e^{2\pi inx}=2\pi n\, e^{2\pi inx},\quad n\in\mathbb{Z},
$$
which by Fourier theory form a complete basis for $\H$. Hence the 
non-negative spectral projection $P_{+}$ associated to $\D_0$ is 
the projection onto $\ol{\Span}\{e^{2\pi in x}:n\geq0\}$.

Since $\D_{0}$ has compact resolvent and is self-adjoint, any failure to obtain a Fredholm
module (and so $K$-homology class) must arise from some function $f\in C([0,1])$ having
non-compact commutator with $F:=\D_{0}(1+\D_{0}^2)^{-1/2}$. 
Indeed this is the case, and to see this
let $x$ be the identity function on $[0,1]$, which generates $C([0,1])$ along with the 
constant functions.  Lemma \ref{lem:compact-comms} shows that to prove that $[F,x]$ 
is not compact, it suffices to prove that $P_+xP_-$ is not compact.

Elementary Fourier theory shows that for $ \sum_{n\in\Z}f_ne^{2\pi inx}\in L^2([0,1])$
$$
x\cdot\sum_{n\in\Z}f_n\,e^{2\pi i n x}=\sum_{n,l\in\Z} f_n\left(\frac{1-\delta_{\ell n}}{2\pi i(n-\ell)}+\frac{1}{2}\delta_{\ell n}\right)\,e^{2\pi i \ell x}.
$$
With $P_+$ the non-negative spectral projection associated to $D_{0}$ 
and $P_-=1-P_+$, we find that
$$
P_+xP_-\cdot\sum_{n\in\Z}f_n\,e^{2\pi i n x}
=\frac{-1}{2\pi i}\sum_{n\geq 1,\,\ell\geq 0}\frac{f_{-n}}{n+\ell}\,e^{2\pi i \ell x}.
$$
Then for $m\in\mathbb{N}$ we define the sequence of vectors
$$
\xi_m=\sum_{n=1}^\infty\frac{\sqrt{m}}{n+m}\,e^{-2\pi i n x}.
$$
\begin{lemma}\label{lem:boundedsequence1}
The sequence $\{\xi_m\}_{m=1}^\infty$ is bounded.
\end{lemma}
\begin{proof}
We have
\begin{align*}
\|\xi_m\|^2=m\sum_{n=1}^\infty\frac{1}{(m+n)^2}=m\psi^{(1)}(m+1),
\end{align*}
where $\psi^{(k)}(x)=(d^{k+1}/dx^{k+1})(\log(\Gamma))(x)$ is the polygamma function of order 
$k$. As $m\rightarrow\infty$, $(m+1)\psi^{(1)}(m+1)\to 1$, so
\begin{align*}
\lim_{m\rightarrow\infty}\|\xi_m\|^2&=\lim_{m\rightarrow\infty}m\cdot\frac{1}{m+1}=1.&&\qedhere
\end{align*}
\end{proof}

With $\zeta_m=P_+xP_-\xi_m$ and $\psi^{(0)}(x)=(d/dx)(\log(\Gamma))(x)$ the digamma
function, we find that
\begin{align}
\Vert\zeta_m\Vert^2&=\frac{m}{4\pi^2}\sum_{\ell=0}^\infty\left(\sum_{n=1}^\infty\frac{1}{(n+m)(n+\ell)}\right)^2\geq\frac{m}{4\pi^2}\sum_{\ell=0}^{m-1}\left(\sum_{n=1}^\infty\frac{1}{(n+m)(n+\ell)}\right)^2\notag\\
&=\frac{m}{4\pi^2}\sum_{\ell=0}^{m-1}\left(\frac{\psi^{(0)}(m+1)-\psi^{(0)}(\ell+1)}{m-\ell}\right)^2=\frac{m}{4\pi^2}\sum_{\ell=0}^{m-1}\frac{1}{(m-\ell)^2}\left(\sum_{k=0}^{m-\ell-1}\frac{1}{\ell+k+1}\right)^2\notag\\
&\geq\frac{m}{4\pi^2}\sum_{\ell=0}^{m-1}\frac{1}{(m-\ell)^2}\left(\frac{m-\ell}{\ell+(m-\ell-1)+1}\right)^2=\frac{m}{4\pi^2}\sum_{\ell=0}^{m-1}\frac{1}{m^2}=\frac{1}{4\pi^2}.
\label{eq:nail-in-coffin}
\end{align}

\begin{lemma}\label{lem:convergezero1}
If $\{\zeta_m\}_{m=1}^\infty$ has a norm convergent 
subsequence $\{\zeta_{m_j}\}_{j=1}^\infty$, then $\zeta_{m_j}\rightarrow0$.
\end{lemma}
\begin{proof}
We show that $\lim_{m\rightarrow\infty}\Ideal{\zeta_m\,|\,e^{2\pi i px}}=0$ for all 
$p\in\mathbb{\Z}$, which shows that if $\zeta_{m_j}\rightarrow\zeta$, then $\zeta=0$. We have
\begin{align*}
\Ideal{\zeta_m\,|\,e^{2\pi i px}}=\left\{\begin{array}{ll} \sum_{n=1}^\infty\frac{-\sqrt{m}}{2\pi i(m+n)(n+p)} & p\geq 0\\ 0 & {\rm otherwise}.\end{array}\right.
\end{align*}
Thus we can ignore the case $p<0$. Computing further gives
\begin{align*}
\Ideal{\zeta_m\,|\,e^{2\pi i px}}=\left\{\begin{array}{ll} \frac{-\sqrt{m}}{2\pi i}\left(\frac{\psi^{(0)}(m+1)-\psi^{(0)}(p+1)}{m-p}\right) & p\geq 0,\ p\neq m\\  \frac{-\sqrt{m}}{2\pi i}\psi^{(1)}(m+1)& p=m.\end{array}\right.
\end{align*}
Since $\psi^{(0)}(m+1)\sim \log(m+1)$  as $m\to\infty$, 
we see that in all cases $\Ideal{\zeta_m\,|\,e^{2\pi i p x}}\to 0$ as $m\to\infty$.
\end{proof}
\begin{corl} The sequence
$\{\zeta_m\}_{m=1}^\infty$ has no norm convergent subsequences.
\end{corl}
\begin{proof}
If $\zeta_m$ had a convergent subsequence $\{\zeta_{m_j}\}_{j=1}^\infty$, 
then $\zeta_{m_j}\rightarrow0$ by Lemma \ref{lem:convergezero1}. 
But by Equation \eqref{eq:nail-in-coffin}, $\|\zeta_{m_j}\|\nrightarrow0$, which is a contradiction.
\end{proof}
\begin{corl} The operator
$P_+xP_-$ is not compact.
\end{corl}
\begin{proof}
By Lemma \ref{lem:boundedsequence1}, $\{\xi_m\}_{m=1}^\infty$ is bounded, but 
$\{P_+xP_-\xi_m\}_{m=1}^\infty$ contains no convergent subsequence.
Hence
$P_+xP_-$ is not compact.
\end{proof}
In summary we have shown the following:
\begin{prop} The self-adjoint extension $\D_{0}$ of the closed symmetric 
operator $\D$ has compact resolvent, and for all $a\in C^\infty([0,1])$, the 
commutators $[\D_{0},a]$ are defined on 
${\rm dom}\,\D$, and are bounded on this dense subset. 
The bounded transform $F:=\D_{0}(1+\D_{0}^{2})^{-\frac{1}{2}}$ 
has the property that the commutator $[F,x]$ is not a compact operator. 
Therefore $(C([0,1]), L^{2}([0,1]), F)$ does not define a Fredholm module.
\end{prop}

\subsection{Infinite deficiency indices: the unit disc}

The next three subsections produce counterexamples to three statements appearing 
in \cite[Theorems 1.3, 6.2]{Valette}. These theorems rely on 
both the finite deficiency index case, and the 
extended definition in \cite[Definition 6.3]{Valette}, which allows for symmetric 
operators having infinite (and equal) deficiency indices. 
The third of the counterexamples below again shows 
that the definition of spectral triple using condition 1.' in place of condition 1. does 
not guarantee that we obtain a Fredholm module.

The counterexamples below will be described using a single basic example. 
For this we let $\mathbb{D}$
be the closed unit disc in $\R^2$, and take
the Hilbert space $L^2(\mathbb{D},\C^2)$ with the measure
$$
C(\mathbb{D})\ni f\mapsto \frac{1}{2\pi}\int_0^{2\pi}\!\!\!\int_0^1 f(r,\theta)\,r\,dr\,d\theta.
$$
Write $\mathring{\mathbb{D}}:=\mathbb{D}\setminus\partial\mathbb{D}$ for the interior of $\mathbb{D}$. We will use the Dirac operator  on $\mathring{\mathbb{D}}$ for our example. This is a densely defined symmetric operator on $L^2(\mathbb{D},\mathbb{C}^2)$, which is given in local polar coordinates by
$$
\D_c:=\begin{pmatrix}0&e^{-i\theta}(-\de_r+ir^{-1}\de_\theta)\\e^{i\theta}(\de_r+ir^{-1}\de_\theta)&0\end{pmatrix}: C_{c}^{\infty}(\mathring{\mathbb{D}}, \C^{2})\rightarrow  C_{c}^{\infty}(\mathring{\mathbb{D}},\C^{2}).
$$ 
Let $\D$ be the closure of $\D_c$,
and observe that its domain is given by
$$
{\rm dom}\,\D=\{f\in L^{2}(\mathbb{D},\C^{2}): \exists f_{n}\in C_{c}^{\infty}(\mathring{\mathbb{D}},\C^2), \quad f_{n}\rightarrow f,\quad \D_c f_{n}\rightarrow g\in L^{2}(\mathbb{D},\C^{2})\}.
$$
This is also referred to as the \emph{minimal domain} (or minimal extension) of the Dirac operator.

 The \emph{maximal domain} (or maximal extension) of the Dirac operator 
 is the domain of its adjoint $\D^{*}$. This extension can be  
 described using distributions. The symmetric operator $\D_c$ induces a dual operator 
\[ 
\D_c^{\dag}:C_{c}^{\infty}(\mathring{\mathbb{D}}, \C^{2})^{\dag}\rightarrow C_{c}^{\infty}(\mathring{\mathbb{D}}, \C^{2})^{\dag},
\]
on the space of distributions $C_{c}^{\infty}(\mathring{\mathbb{D}}, \C^{2})^{\dag}$, 
uniquely determined by the formula 
\[
\langle \D_c^{\dag}\phi,f\rangle:=\langle \phi, \D_c f\rangle,\qquad \phi\in C_c^\infty(\mathring{\mathbb{D}}, \C^{2})^{\dag},\ f\in C_{c}^{\infty}(\mathring{\mathbb{D}}, \C^{2}).
\]
A similar formula embeds $L^{2}(\mathbb{D},\C^{2})$ into the space of distributions. Using these identifications, the domain of $\D^*$ is given by
\[
{\rm dom}\, \D^{*}=\{f\in L^{2}(\mathbb{D},\C^{2}): \D_c^{\dag}f\in L^{2}(\mathbb{D},\C^{2})\}.
\] 
The domain of $\D^*$ coincides with the first Sobolev space 
$H^1(\mathbb{D},\C^2)$, \cite[Proposition 20.7]{BBW}. With this characterisation
it is straightforward to check that
for any smooth bounded function $a$ on the disc, $a:{\rm dom}\,\D\to{\rm dom}\,\D$
and $a:{\rm dom}\,\D^*\to{\rm dom}\,\D^*$, and $[\D^*,a]$ is bounded on both ${\rm dom}\,\D$
and ${\rm dom}\,\D^*$. 
\begin{lemma} The operator $(1+\D^*\D)^{-1/2}$ is compact.
\end{lemma}
\begin{proof}
The eigenvectors of $\D^*\D$ are
\begin{align*}
\left\{\begin{pmatrix}J_n(r\alpha_{n,k})e^{in\theta}\\0\end{pmatrix},\quad\begin{pmatrix}0\\J_n(r\alpha_{n,k})e^{in\theta}\end{pmatrix}:n\in\mathbb{Z},k=1,2,\ldots\right\},
\end{align*}
where $\alpha_{n,k}$ denotes the $k$-th positive root of the Bessel function $J_n$.
These eigenvectors are complete for $L^2(\mathbb{D},\mathbb{C}^2)$ by 
arguments similar to those in section \ref{sec:commutator}: 
namely $\{e^{in\theta}:n\in\mathbb{Z}\}$ is complete for $S^1$, 
and $\{J_n(r\alpha_{n,k}):k\geq1\}$ is complete for $L^2([0,1],r\,dr)$ for all $n\in\mathbb{Z}$,
\cite{BoasPollard}.

We note that
\begin{align*}
\D^*\D \begin{pmatrix}J_n(r\alpha_{n,k})e^{in\theta}\\0\end{pmatrix}=\alpha_{n,k}^2\begin{pmatrix}J_n(r\alpha_{n,k})e^{in\theta}\\0\end{pmatrix},\\
\D^*\D\begin{pmatrix}0\\J_n(r\alpha_{n,k})e^{in\theta}\end{pmatrix}=\alpha_{n,k}^2\begin{pmatrix}0\\J_n(r\alpha_{n,k})e^{in\theta}\end{pmatrix},
\end{align*}
so the eigenvalues of $\D^*\D$ are $\{\alpha_{n,k}^2\}_{n=0,k=1}^\infty$. 
Each of these eigenvalues has multiplicity 4. Since 
$\alpha_{n,k}\rightarrow\infty$ as $n,k\rightarrow\infty$, 
it follows that $(1+\D^*\D)^{-1/2}$ is compact.
\end{proof}
Since $(1+\D^*\D)^{-1/2}$ is compact,
the data $(C^{\infty}(\mathbb{D}),L^{2}(\mathbb{D},\C^{2}),\D)$ satisfies 
the definition of symmetric unbounded Fredholm module in \cite[Definition 6.3]{Valette}.
The closed symmetric operator $\D$ has infinite deficiency indices, since one may check directly that
\begin{align*}
\ker(\D^*\mp i)\supset\ol{\Span}\left\{\begin{pmatrix}\pm ie^{in\theta}I_n(r)\\e^{i(n+1)\theta}I_{n+1}(r)\end{pmatrix},\begin{pmatrix}\pm ie^{-i(n+1)\theta}I_{n+1}(r)\\e^{-in\theta}I_n(r)\end{pmatrix}:\,n\in\mathbb{N}\right\},
\end{align*}
where the $I_n$ are modified Bessel functions of the first kind. Thus $\D$ has self-adjoint extensions.
It is a well known general fact that any closed symmetric extension $\D_{ext}$ of $\D$ must 
satisfy ${\rm dom}\,\D\subset{\rm dom}\,\D_{ext}\subset{\rm dom}\,\D^{*}$, \cite{RS}.

\subsection{An example with noncompact resolvent}
\label{subsec:resolvent}

The arguments in the proofs of \cite[Theorems 1.2 and 6.2]{Valette} 
purport to show that all self-adjoint extensions of an
operator such as $\D$ above give rise to a Fredholm module 
(for $C^\infty(\mathbb{D})$ in this example). As in the finite deficiency index case, this fails,
but it can fail in more ways.

%

The issue of (relatively) compact resolvent is addressed  on \cite[page 198]{Valette}. 
The assertions about extensions used there are false\footnote{There are no non-trivial
self-adjoint extensions of a self-adjoint operator.}, and we now 
show how to obtain an extension with noncompact resolvent. Write
$$
\D
=\begin{pmatrix} 0 & \D_-\\ \D_+ & 0\end{pmatrix}.
$$
Then define a self-adjoint extension of $\D$ by 
$$
\D_{ext}:=\begin{pmatrix}0 & \D_+^*\\ \D_+ & 0\end{pmatrix},
$$
where $\D_+=(\D_+)_{min}$ is the minimal extension, and $((\D_+)_{min})^*=(\D_-)_{max}$
is the maximal extension of $\D_-$, \cite[Proposition 20.7]{BBW}.
As in Equation \eqref{eq:star} in the next section, it is easily 
checked that
\[\ker (\D_{-})_{max}=\{r^{n}e^{-in\theta}: n=0,1,\dots\},\]
thus $\D_{ext}$ has infinite dimensional kernel and so the resolvent is not compact. 
As the constant function $1\in C^\infty(\mathbb{D})$ acts as the identity on the 
Hilbert space, this shows that we fail to obtain a spectral triple for 
$C^\infty(\mathbb{D})$. Since this also means that 
$$
(1-F_{\D_{ext}}^2)=(1-\D_{ext}^2(1+\D_{ext}^2)^{-1})=(1+\D_{ext}^2)^{-1}
$$
is not compact, we do not obtain a Fredholm module for $C(\mathbb{D})$.

\subsection{The dependence of $K$-homology classes on the choice of extension}
Next we show that the claim in \cite[Theorem 6.2]{Valette} that the $K$-homology class
of a symmetric operator with equal deficiency indices
is independent of the self-adjoint extension is false. This example also shows 
that \cite[Theorem 1.3]{Valette} is false.

To define our self-adjoint extensions, we use boundary conditions. The trace theorem, 
\cite[Theorem 11.4]{BBW},
gives the continuity of $f\mapsto f|_{\partial\mathbb{D}}$ as a map 
${\rm dom}\,\D^*\to H^{1/2}(\partial\mathbb{D},\C^2)\subset L^2(S^1,\C^2)$. Thus we can 
use the boundary values to specify domains of extensions of $\D$ inside ${\rm dom}\,\D^*$. 

We consider APS-type  extensions arising
from the  projections $P_N:L^2(S^1)\to L^2(S^1)$, $N\in\mathbb{Z}$, defined by
$$
P_N\left(\sum_{k\in\Z}c_ke^{ik\theta}\right)=\sum_{k\geq N}c_ke^{ik\theta},
\quad\quad \sum_{k\in\Z}c_ke^{ik\theta}\in L^2(S^1).
$$
We use $P_N$ to 
define self-adjoint extensions by setting
\begin{align*}
{\rm dom}\,\D_{P_N}&:=\left\{\begin{pmatrix}\xi_1\\ \xi_2\end{pmatrix}\in{\rm dom}\,\D^*:\,  P_N(\xi_1|_{\de\mathbb{D}})=0,\ (1-P_{N+1})(\xi_2|_{\de\mathbb{D}})=0      \right\}\\
&\D_{P_N}\xi:=\D^*\xi,\qquad \mbox{for}\ \xi\in{\rm dom}\,\D_{P_N}.
\end{align*}

The self-adjoint extensions above
{\em do} define Fredholm modules and so $K$-homology 
classes for the algebra of functions constant on the boundary, since these
functions preserve the domain, but each $\D_{P_N}$
defines a different class. This is easy, and not new: see \cite[Appendix A]{BDT}, 
since the index (that is the pairing of the $K$-homology class with the 
constant function 1) is easily computed to be
$$
\textrm{Index}((\D_{P_N})_+)= N.
$$
The reason is that
\begin{equation}
\ker(\D^*)=\ol{\Span}\left\{\begin{pmatrix}r^ne^{in\theta}\\0\end{pmatrix},\begin{pmatrix}0\\r^ne^{-in\theta}\end{pmatrix}:n=0,1,2,\dots\right\},
\label{eq:star}
\end{equation}
and so
\begin{align*}
\ker((\D_{P_N})_+)&=\left\{\begin{array}{ll}\{0\}&N\leq0\\\ol{\Span}\{r^ne^{in\theta}:0\leq n<N\}&N>0,\end{array}\right.
\end{align*}
whilst
\begin{align*}
\ker((\D_{P_N})_-)&=\left\{\begin{array}{ll}\{0\}&N>-1\\\ol{\Span}\{r^ne^{-in\theta}:0\leq n\leq-N-1\}&N\leq-1.\end{array}\right.
\end{align*}

\subsection{Another noncompact commutator}
\label{sec:commutator}

In subsection \ref{subsec:finite} we showed that the weakened definition of spectral triple does not
suffice to guarantee that we obtain a Fredholm module. The example there also
showed that \cite[Theorem 1.2]{Valette} is false. Now we show that the  problem of noncompact commutators 
persists in the infinite deficiency index case. This shows that \cite[Theorem 6.2]{Valette}
can not be repaired by requiring that the self-adjoint extensions employed  have compact resolvents.

%
%
%

In this section, $\D_P$ shall denote the self-adjoint extension $\D_{P_0}$. As $\D_{P}$ is an extension of $\D$, we find that $[\D_{P},a]$ is defined and bounded on
the domain of $\D$, for all $a\in C^\infty(\mathbb{D})$. As in subsection \ref{subsec:finite}, 
we need to compute commutators with the phase of $\D_P$.

For $k\geq 1$, let $\alpha_{n,k}$ denote the $k\Th$ positive zero of the Bessel function
$J_n$. Then the  eigenvectors of $\D_P^2$ are
\begin{align}\label{eigenvectors}
&\left\{\begin{pmatrix}J_n(r\alpha_{n-1,k})e^{-in\theta}\\0\end{pmatrix},\begin{pmatrix}0\\J_n(r\alpha_{n-1,k})e^{in\theta}\end{pmatrix}\right\}_{n,k=1}^\infty,\nonumber\\ &\left\{\begin{pmatrix}J_n(r\alpha_{n,k})e^{in\theta}\\0\end{pmatrix},\begin{pmatrix}0\\J_n(r\alpha_{n,k})e^{-in\theta}\end{pmatrix}\right\}_{n=0,k=1}^\infty.
\end{align}
\begin{lemma}\label{Dsquare} The eigenvectors \eqref{eigenvectors} of $\D_{P}^{2}$ span $L^2(\mathbb{D},\mathbb{C}^2)$. The corresponding set of eigenvalues is $\{\alpha_{n,k}^2\}_{n=0,k=1}^\infty$, and hence the resolvent of $\D_{P}$ is compact.
\end{lemma}
\begin{proof} With the measure $rdrd\theta$, we can take $\mathbb{D}=[0,1]\times S^1/\sim$, where $\sim$ is the identification $(0,z)\sim(0,1)$ for $z\in S^1$. It is well known that $\{e^{in\theta}\}_{n=-\infty}^\infty$ is complete for $L^2(S^1)$, so it is enough to show that

(a) $\{r\mapsto J_n(r\alpha_{n-1,k})\}_{k=1}^\infty$ spans $L^2([0,1],r\,dr)$ for all $n=1,2,\ldots$, and

(b) $\{r\mapsto J_n(r\alpha_{n,k})\}_{k=1}^\infty$ spans $L^2([0,1],r\,dr)$ for all $n=0,1,2,\dots$.

Statement (a) is true by \cite[Theorem 6]{BoasPollard}, and (b) is true by \cite[Theorem 2]{BoasPollard}.\footnote{In \cite{BoasPollard}, Boas and Pollard take the usual measure on $[0,1]$ instead of $r\,dr$ and a slightly different set of functions, but it is easy to see that the two approaches are equivalent.} Hence the eigenfunctions above are the entire set of eigenfunctions, and the set of eigenvalues is $\{\alpha_{n,k}^2\}_{n=0,k=1}^\infty$. Each of these eigenvalues has multiplicity 4. In particular
$\D_P$ has no kernel, and
since $\alpha_{n,k}\rightarrow\infty$ as $n,k\rightarrow\infty$, $(1+\D_P^2)^{-1/2}$ is compact.
\end{proof}

To facilitate our computations we now describe an orthonormal eigenbasis for $\D_{P}$.
\begin{prop} The vectors
\begin{align*}
\Ket{1,n,k,\pm}&=\frac{1}{J_n(\alpha_{n-1,k})}\begin{pmatrix}J_n(r\alpha_{n-1,k})e^{-in\theta}\\\pm J_{n-1}(r\alpha_{n-1,k})e^{-i(n-1)\theta}\end{pmatrix},\\
\Ket{2,n,k,\pm}&=\frac{1}{J_n(\alpha_{n-1,k})}\begin{pmatrix}J_{n-1}(r\alpha_{n-1,k})e^{i(n-1)\theta}\\\mp J_n(r\alpha_{n-1,k})e^{in\theta}\end{pmatrix},
\end{align*}
$n,k=1,2,\ldots$.
form a normalised complete set of eigenvectors for $\D_{P}$. The corresponding set of eigenvalues is given by 
\[\D_{P}\Ket{j,n,k,\pm} =\pm \alpha_{n-1,k}\Ket{j,n,k,\pm}.\]
\end{prop}
\begin{proof}
From Lemma \ref{Dsquare} it is straightforward to show that the eigenvectors and eigenvalues of $\D_P$ are
\begin{align*}
\D_P\begin{pmatrix}J_n(r\alpha_{n-1,k})e^{-in\theta}\\\pm J_{n-1}(r\alpha_{n-1,k})e^{-i(n-1)\theta}\end{pmatrix}&=\pm\alpha_{n-1,k}\begin{pmatrix}J_n(r\alpha_{n-1,k})e^{-in\theta}\\\pm J_{n-1}(r\alpha_{n-1,k})e^{-i(n-1)\theta}\end{pmatrix},\\
\D_P\begin{pmatrix}J_{n-1}(r\alpha_{n-1,k})e^{i(n-1)\theta}\\\mp J_n(r\alpha_{n-1,k})e^{in\theta}\end{pmatrix}&=\pm\alpha_{n-1,k}\begin{pmatrix}J_{n-1}(r\alpha_{n-1,k})e^{i(n-1)\theta}\\\mp J_n(r\alpha_{n-1,k})e^{in\theta}\end{pmatrix},
\end{align*}
for $n,k=1,2,\ldots$. Note that these eigenvectors are complete for $L^2(\mathbb{D},\mathbb{C}^2)$ since we can recover our spanning set \eqref{eigenvectors} from linear combinations of these. 

To normalise these eigenvectors, we use the following standard integrals which can be found in \cite{W}:
\begin{align*}
&\Ideal{\begin{pmatrix}J_n(r\alpha_{n-1,k})e^{-in\theta}\\\pm J_{n-1}(r\alpha_{n-1,k})e^{-i(n-1)\theta}\end{pmatrix},\begin{pmatrix}J_n(r\alpha_{n-1,k})e^{-in\theta}\\\pm J_{n-1}(r\alpha_{n-1,k})e^{-i(n-1)\theta}\end{pmatrix}}\\
&=\frac{1}{2\pi}\int_0^{2\pi}\!\!\!\int_0^1(J_n^2(r\alpha_{n-1,k})+J_{n-1}^2(r\alpha_{n-1,k}))r\,dr\,d\theta\\
&=\frac{1}{2}\left(J_n^2(\alpha_{n-1,k})+J_n^2(\alpha_{n-1,k})\right)=J_n^2(\alpha_{n-1,k}),
\end{align*}
and similarly
\begin{align*}
\Ideal{\begin{pmatrix}J_{n-1}(r\alpha_{n-1,k})e^{i(n-1)\theta}\\\pm J_n(r\alpha_{n-1,k})e^{in\theta}\end{pmatrix},\begin{pmatrix}J_{n-1}(r\alpha_{n-1,k})e^{i(n-1)\theta}\\\pm J_n(r\alpha_{n-1,k})e^{in\theta}\end{pmatrix}}&=J_n^2(\alpha_{n-1,k}).\qedhere
\end{align*}
\end{proof}

Our purpose is to find a function $a\in C(\mathbb{D})$ for which the commutator $[F,a]$ is not compact, where $F=\D_P(1+\D_P^2)^{-1/2}$ is the bounded transform. Let $P_+$ be the non-negative spectral projection associated to $\D_P$, and let $P_-=1-P_+$. By Lemma \ref{lem:compact-comms}, we need only show that there is some $a\in C(\mathbb{D})$
for which the operator $P_+aP_-$ is not compact.

In terms of the eigenbasis of $\D_P$, for any $a\in C(\mathbb{D})$ we can write
\begin{align}\label{pap}
P_+aP_-&=\sum_{i,j=1,2}\sum_{n,m,k,\ell=1}^\infty
\Ket{i,n,k,+}\Bra{i,n,k,+}a\Ket{j,m,\ell,-}\Bra{j,m,\ell,-}.
\end{align}
Now we fix $a=re^{-i\theta}$. The function $re^{-i\theta}$ generates 
$C(\mathbb{D})$ (along with the constant function 1), and 
fails to preserve the domain of $\D_P$; for instance $re^{-i\theta}\cdot\Ket{2,1,k,\pm}\notin\dom(\D_P)$. To show that $P_+re^{-i\theta}P_-$ is not compact, we will construct a bounded sequence of vectors $\xi_{n}$, with the property that $P_{+}re^{-i\theta}P_{-}$ maps $\xi_n$ to a sequence with no convergent 
subsequences. In order to find the sequence $\xi_{n}$, we first derive an explicit formula for $P_{+}re^{-i\theta}P_{-}$.

\begin{lemma}The operator $P_+re^{-i\theta}P_-$ can be expressed as
\begin{align*}
P_+re^{-i\theta}P_-&=\sum_{m,k,\ell=1}^\infty\frac{2\alpha_{m,k}}{(\alpha_{m,k}-\alpha_{m-1,\ell})(\alpha_{m,k}+\alpha_{m-1,\ell})^2}\Ket{1,m+1,k,+}\Bra{1,m,\ell,-}\\
&\quad+\sum_{n,k,\ell=1}^\infty\frac{2\alpha_{n,\ell}}{(\alpha_{n-1,k}-\alpha_{n,\ell})(\alpha_{n,\ell}+\alpha_{n-1,k})^2}\Ket{2,n,k,+}\Bra{2,n+1,\ell,-}\\
&\quad+\sum_{k\neq\ell}\frac{1}{\alpha_{0,k}+\alpha_{0,\ell}}\Ket{1,1,k,+}\Bra{2,1,\ell,-}+\sum_{k=1}^\infty\frac{1}{\alpha_{0,k}}\Ket{1,1,k,+}\Bra{2,1,k,-}.
\end{align*}
\end{lemma}
\begin{proof} In view of Equation \eqref{pap}, we first compute the operators  $\Bra{i,n,k,+}re^{-i\theta}\Ket{j,m,\ell,-}$ for $i,j=1,2$.
Using integration by parts and standard recursion relations and 
identities for the Bessel functions and their derivatives, \cite{W}, we find:
\begin{enumerate}
\item Case $i=j=1$:
\begin{align*}
&\Bra{1,n,k,+}re^{-i\theta}\Ket{1,m,\ell,-}\\
&=\frac{1}{2\pi J_n(\alpha_{n-1,k})J_m(\alpha_{m-1,\ell})}\int_0^{2\pi}\!\!\!\int_0^1r^2e^{i(n-m-1)\theta}\big(J_n(r\alpha_{n-1,k})J_m(r\alpha_{m-1,\ell})\\
&\qquad\qquad\qquad\qquad\qquad\qquad\qquad\qquad-J_{n-1}(r\alpha_{n-1,k})J_{m-1}(r\alpha_{m-1,\ell})\big)\,dr\,d\theta\\
&=\frac{\delta_{n,m+1}}{J_{m+1}(\alpha_{m,k})J_m(\alpha_{m-1,\ell})}\int_0^1r^2J_{m+1}(r\alpha_{m,k})J_m(r\alpha_{m-1,\ell})-r^2J_m(r\alpha_{m,k})J_{m-1}(r\alpha_{m-1,\ell})\,dr\\
&=\frac{2\alpha_{m,k}\delta_{n,m+1}}{(\alpha_{m,k}-\alpha_{m-1,\ell})(\alpha_{m,k}+\alpha_{m-1,\ell})^2};
\end{align*}
\item Case $i=1,j=2$:
\begin{align*}
&\Bra{1,n,k,+}re^{-i\theta}\Ket{2,m,\ell,-}\\
&=\frac{1}{2\pi J_n(\alpha_{n-1,k})J_m(\alpha_{m-1,\ell})}\int_0^{2\pi}\!\!\!\int_0^1r^2e^{i(m+n-2)\theta}\big(J_n(r\alpha_{n-1,k})J_{m-1}(r\alpha_{m-1,\ell})\\
&\qquad\qquad\qquad\qquad\qquad\qquad\qquad\qquad+J_{n-1}(r\alpha_{n-1,k})J_m(r\alpha_{m-1,\ell})\big)\,dr\,d\theta\\
&=\left\{\begin{array}{cl}\frac{1}{J_1(\alpha_{0,k})J_1(\alpha_{0,\ell})}\int_0^1r^2J_1(r\alpha_{0,k})J_0(r\alpha_{0,\ell})+r^2J_0(r\alpha_{0,k})J_1(r\alpha_{0,\ell})\,dr&n=m=1\\0&\text{otherwise}\end{array}\right.\\
&=\left\{\begin{array}{cl}\frac{1}{\alpha_{0,k}+\alpha_{0,\ell}}&n=m=1\text{ and }k\neq\ell\\\frac{1}{\alpha_{0,k}}&n=m=1\text{ and }k=\ell\\0&\text{otherwise;}\end{array}\right.
\end{align*}
\item Case $i=2,j=1$:
\begin{align*}
&\Bra{2,n,k,+}re^{-i\theta}\Ket{1,m,\ell,-}\\
&=\frac{1}{2\pi J_n(\alpha_{n-1,k})J_m(\alpha_{m-1,\ell})}\int_0^{2\pi}\!\!\!\int_0^1r^2e^{-i(n+m)\theta}\big(J_{n-1}(r\alpha_{n-1,k})J_m(r\alpha_{m-1,\ell})\\
&\qquad\qquad\qquad\qquad\qquad\qquad\qquad\qquad+J_n(r\alpha_{n-1,k})J_{m-1}(r\alpha_{m-1,k\ell})\big)\,dr\,d\theta\\
&=0;
\end{align*}
\item Case $i=j=2$:
\begin{align*}
&\Bra{2,n,k,+}re^{-i\theta}\Ket{2,m,\ell,-}\\
&=\frac{1}{2\pi J_n(\alpha_{n-1,k})J_m(\alpha_{m-1,\ell})}\int_0^{2\pi}\!\!\!\int_0^1r^2e^{i(m-n-1)\theta}\big(J_{n-1}(r\alpha_{n-1,k})J_{m-1}(r\alpha_{m-1,\ell})\\
&\qquad\qquad\qquad\qquad\qquad\qquad\qquad\qquad-J_n(r\alpha_{n-1,k})J_m(r\alpha_{m-1,\ell})\big)\,dr\,d\theta\\
&=\frac{\delta_{m,n+1}}{J_n(\alpha_{n-1,k})J_{n+1}(\alpha_{n,\ell})}\int_0^1r^2J_{n-1}(r\alpha_{n-1,k})J_n(r\alpha_{n,\ell})-r^2J_n(r\alpha_{n-1,k})J_{n+1}(r\alpha_{n,\ell})\,dr\\
&=\frac{2\alpha_{n,\ell}\delta_{m,n+1}}{(\alpha_{n-1,k}-\alpha_{n,\ell})(\alpha_{n,\ell}+\alpha_{n-1,k})^2}.
\end{align*}
\end{enumerate}
The desired equation is now obtained by using these cases in combination with \eqref{pap}.
\end{proof}
For convenience we write 
$$\Ket{\ell,-}:=\Ket{2,1,\ell,-},\qquad\Ket{k,+}:=\Ket{1,1,k,+},$$
and define the sequence
$$\xi_n:=\sum_{\ell=1}^\infty\frac{\sqrt{n}}{n+\ell}\Ket{\ell,-},\quad n=1,2,\ldots.$$
\begin{lemma}\label{lem:boundedsequence}
The sequence $\{\xi_n\}_{n=1}^\infty$ is bounded.
\end{lemma}
\begin{proof}
As in Lemma \ref{lem:boundedsequence1} we have
\begin{align*}
\|\xi_n\|^2=n\sum_{\ell=1}^\infty\frac{1}{(n+\ell)^2}=n\psi^{(1)}(n+1),
\end{align*}
where $\psi^{(m)}(x)=(d^{m+1}/dx^{m+1})(\log(\Gamma))(x)$ is the polygamma 
function of order $m$. As $n\rightarrow\infty$, $(n+1)\psi^{(1)}(n+1)\to 1$, so $\|\xi_n\|^2\to 1$.
\end{proof}
 To simplify the computations, we subtract the operator
$$
K:=\sum_{k=1}^\infty\frac{1}{2\alpha_{0,k}}\Ket{1,1,k,+}\Bra{2,1,k,-}
$$
from $P_+re^{-i\theta}P_-$, since $K$ is obviously compact. 
To this end, define
$$\zeta_n:=(P_+re^{-i\theta}P_--K)\xi_n.$$
Our purpose is to show that $\zeta_{n}$ has no convergent subsequence. To this end we investigate its limiting behaviour.
\begin{lemma}\label{lem:liminf}
$$\liminf_{n\rightarrow\infty}\|\zeta_n\|\geq\frac{1}{2\pi}.$$
\end{lemma}
\begin{proof}
We have
\begin{align*}
\zeta_n&=\sum_{k,\ell=1}^\infty\frac{\sqrt{n}}{(n+\ell)(\alpha_{0,k}+\alpha_{0,\ell})}\Ket{k,+}.
\end{align*}
It is proved in \cite[Lemma 1]{LM} that for all $\ell\geq 1$,
\begin{align}\label{eq:bounds}
\pi(\ell-1/4)<\alpha_{0,\ell}<\pi(\ell-1/8),
\end{align}
yielding the inequality
\begin{align*}
\frac{\sqrt{n}}{(n+\ell)(\alpha_{0,k}+\alpha_{0,\ell})}>\frac{\sqrt{n}}{(n+\ell)(\alpha_{0,k}+\pi(\ell-1/8))}.
\end{align*}
This allows us to estimate the coefficients of $\zeta_n$ via
\begin{align*}
\sum_{\ell=1}^\infty\frac{\sqrt{n}}{(n+\ell)(\alpha_{0,k}+\alpha_{0,\ell})}&\geq\sum_{\ell=1}^\infty\frac{\sqrt{n}}{(n+\ell)(\alpha_{0,k}+\pi(\ell-1/8))}\\
&=\frac{\sqrt{n}}{\pi(n-\alpha_{0,k}/\pi+1/8)}\sum_{\ell=1}^\infty\left(\frac{1}{\ell+\alpha_{0,k}/\pi-1/8}-\frac{1}{n+\ell}\right)\\
&=\frac{\sqrt{n}}{\pi(n-\alpha_{0,k}/\pi+1/8)}\sum_{\ell=1}^\infty\left(\frac{1}{\ell+\alpha_{0,k}/\pi-1/8}-\frac{1}{\ell}+\frac{1}{\ell}-\frac{1}{\ell+n}\right)\\
&=\frac{\sqrt{n}}{\pi(n-\alpha_{0,k}/\pi+1/8)}\left(-\psi^{(0)}(\alpha_{0,k}/\pi+7/8)+\psi^{(0)}(n+1)\right)\\
&=\frac{\sqrt{n}}{\pi}\frac{\psi^{(0)}(n+1)-\psi^{(0)}(\alpha_{0,k}/\pi+7/8)}{n-\alpha_{0,k}/\pi+1/8}
\end{align*}
which allows us to bound $\Vert\zeta_n\Vert$ by
\begin{align}\label{eq:lowerbound1}
\|\zeta_n\|^2&\geq\frac{n}{\pi^2}\sum_{k=1}^\infty\left(\frac{\psi^{(0)}(n+1)-\psi^{(0)}(\alpha_{0,k}/\pi+7/8)}{n-\alpha_{0,k}/\pi+1/8}\right)^2\nonumber\\
&\geq \frac{n}{\pi^2}\sum_{k=1}^n\left(\frac{\psi^{(0)}(n+1)-\psi^{(0)}(\alpha_{0,k}/\pi+7/8)}{n-\alpha_{0,k}/\pi+1/8}\right)^2.
\end{align}
Now, $\alpha_{0,k}/\pi\in(k-1/4,k-1/8)$ by Equation \eqref{eq:bounds}, and $\psi^{(0)}$ increases monotonically on $(0,\infty)$, so for $k\leq n$ we have
\begin{align*}
0\leq\!\psi^{(0)}(n+1)-\psi^{(0)}(k+1)\!<\!\psi^{(0)}(n+1)-\psi^{(0)}(k+3/4)\!<\!\psi^{(0)}(n+1)-\psi^{(0)}(\alpha_{0,k}/\pi+7/8).
\end{align*}
For $k\leq n$,
$$\psi^{(0)}(n+1)-\psi^{(0)}(k+1)=\sum_{j=0}^{n-k-1}\frac{1}{k+j+1},$$
and so
$$0\leq\sum_{j=0}^{n-k-1}\frac{1}{k+j+1}<\psi^{(0)}(n+1)-\psi^{(0)}(\alpha_{0,k}/\pi+7/8).$$
For $k\leq n$ we also have
$$0<n-\alpha_{0,k}/\pi+1/8<n-k+3/8,$$
allowing us to obtain the estimate
\begin{align}
&\sum_{k=1}^{n}\left(\frac{\psi^{(0)}(n+1)-\psi^{(0)}(\alpha_{0,k}/\pi+7/8)}{n-\alpha_{0,k}/\pi+1/8}\right)^2>\sum_{k=1}^{n}\frac{1}{(n-k+3/8)^2}\left(\sum_{j=0}^{n-k-1}\frac{1}{k+j+1}\right)^2\nonumber\\
&\geq\sum_{k=1}^{n}\frac{1}{(n-k+3/8)^2}\cdot\left(\frac{n-k}{k+(n-k-1)+1}\right)^2
=\sum_{k=1}^{n}\frac{(n-k)^2}{(n-k+3/8)^2}\frac{1}{n^2}\nonumber\\
&\geq\sum_{k=1}^{n}\frac{(n-k)^2}{(n-k+1)^2}\frac{1}{n^2}
=\frac{1}{n^2}\sum_{j=2}^{n}\frac{(j-1)^2}{j^2}\geq \frac{1}{n^2}\frac{n-1}{4}.
\label{eq:lowerbound2}
\end{align}
Thus combining Equations \eqref{eq:lowerbound1} and \eqref{eq:lowerbound2}   yields
\begin{align}
\|\zeta_n\|^2\geq \frac{n}{\pi^2}\sum_{k=1}^{n}\left(\frac{\psi^{(0)}(n+1)-\psi^{(0)}(\alpha_{0,k}/\pi+7/8)}{n-\alpha_{0,k}/\pi+1/8}\right)^2\geq\frac{n-1}{4n\pi^2}.
\end{align}
As $n\rightarrow\infty$,
\begin{align*}
\liminf_{n\rightarrow\infty}\|\zeta_n\|^2
\geq \frac{1}{4\pi^2}
&&&\qedhere
\end{align*}
\end{proof}
Next we analyse the possible limits of convergent subsequences of $\zeta_{n}$, should they exist.
\begin{lemma}\label{lem:convergezero}
If $\{\zeta_n\}_{n=1}^\infty$ has a norm convergent subsequence $\{\zeta_{n_j}\}_{j=1}^\infty$, then $\zeta_{n_j}\rightarrow0$.
\end{lemma}
\begin{proof}
We show that $\lim_{n\rightarrow\infty}\Ideal{\zeta_n\,|k,+}=0$ for all $k=1,2,\ldots$, which shows that if $\zeta_{n_j}\rightarrow\zeta$, then $\zeta=0$. We have
\begin{align*}
\Ideal{\zeta_n\,|\,k,+}=\sum_{\ell=1}^\infty\frac{\sqrt{n}}{(n+\ell)(\alpha_{0,k}+\alpha_{0,\ell})}
\end{align*}
Since $\alpha_{0,k}\in(\pi k-\pi/4,\pi k-\pi/8)$ by Equation \eqref{eq:bounds}, we have
$$\frac{1}{\alpha_{0,k}+\alpha_{0,\ell}}<\frac{1}{\pi(k+\ell-1/2)}.$$
Hence
\begin{align*}
0\leq\Ideal{\zeta_n|k,+}&\leq\frac{\sqrt{n}}{\pi}\sum_{\ell=1}^\infty\frac{1}{(n+\ell)(k+\ell-1/2)}
=\frac{\sqrt{n}}{\pi(n-k+1/2)}\sum_{\ell=1}^\infty\left(\frac{1}{k+\ell-1/2}-\frac{1}{n+\ell}\right)\\
&=\frac{\sqrt{n}}{\pi(n-k+1/2)}\left(\psi^{(0)}(n+1)-\psi^{(0)}(k+1/2)\right).
\end{align*}
As $n\rightarrow\infty$, $\psi^{(0)}(n)\sim\ln(n)$. Hence
\begin{align*}
\lim_{n\rightarrow\infty}\frac{\sqrt{n}}{\pi(n-k+1/2)}\left(\psi^{(0)}(n+1)-\psi^{(0)}(k+1/2)\right)&=\lim_{n\rightarrow\infty}\left(\frac{\sqrt{n}(\ln(n+1)-\psi^{(0)}(k+1/2))}{\pi(n-k+1/2)}\right)\\
&=0.
\end{align*}
Hence $\lim_{n\rightarrow\infty}\Ideal{\zeta_n\,|\,k,+}=0$.
\end{proof}
\begin{corl} The sequence
$\{\zeta_n\}_{n=1}^\infty$ has no norm convergent subsequences.
\end{corl}
\begin{proof}
If $\zeta_n$ had a convergent subsequence $\{\zeta_{n_j}\}_{j=1}^\infty$, then $\zeta_{n_j}\rightarrow0$ by Lemma \ref{lem:convergezero}. But by Lemma \ref{lem:liminf}, $\|\zeta_{n_j}\|\nrightarrow0$, which is a contradiction.
\end{proof}
\begin{corl} The operator
$P_+re^{-i\theta}P_-$ is not compact.
\end{corl}
\begin{proof}
By Lemma \ref{lem:boundedsequence}, $\{\xi_n\}_{n=1}^\infty$ is bounded, but 
$\{(P_+re^{-i\theta}P_--K)\xi_n\}_{n=1}^\infty$ contains no convergent subsequence.
As $P_+re^{-i\theta}P_-$ differs from $P_+re^{-i\theta}P_--K$ by a compact operator, 
$P_+re^{-i\theta}P_-$ is not compact.
\end{proof}
In summary we have shown the following:
\begin{prop} The self-adjoint extension $\D_{P}$ 
of the closed symmetric operator $\D$ has compact resolvent, 
and for all $a\in C^\infty(\mathbb{D})$, the commutators $[\D_{P},a]$ are defined on 
${\rm dom}\,\D$, and are bounded on this dense subset. 
The bounded transform $F:=\D_{P}(1+\D_{P}^{2})^{-\frac{1}{2}}$ 
has the property that the commutator $[F,re^{-i\theta}]$ is not a 
compact operator. Therefore $(C(\mathbb{D}), L^{2}(\mathbb{D},\mathbb{C}^2), F)$ 
does not define a Fredholm module.
\end{prop}

\end{document}